\documentclass[11pt]{article}
\usepackage{amsmath}
\usepackage{amssymb}
\usepackage{braket}
\usepackage{url}
\usepackage{fancybox}
\def\C{\mathbb{C}}
\def\RR{\boldsymbol{R}}
\def\BB {\boldsymbol{B}}

\def\FF{\boldsymbol{F}}
\def\SSS{\boldsymbol{S}}

\def\sss{\boldsymbol{s}}
\def\BB {\boldsymbol{B}}
\def\KK {\boldsymbol{K}}

\def\AAA {\boldsymbol{A}}
\def\xx{\boldsymbol{x}}
\def\zz{\boldsymbol{z}}
\def\dd{\boldsymbol{d}}
\def\ee{\boldsymbol{e}}
\def\ff{\boldsymbol{f}}
\def\yy{\boldsymbol{y}}

\def\ee{\boldsymbol{e}}
\def\uu{\boldsymbol{u}}

\def\ww{\boldsymbol{w}}

\def\UU{\boldsymbol{U}}

\def\EE{\boldsymbol{E}}

\def\GG{\boldsymbol{G}}
\def\HH{\boldsymbol{H}}
\def\WW{\boldsymbol{W}}
\def\PP{\boldsymbol{P}}

\def\QQ{\boldsymbol{Q}}
\def\RR{\boldsymbol{R}}
\def\TT{\boldsymbol{T}}

\def\YY{\boldsymbol{Y}}

\def\II{\boldsymbol{I}}

\def\00{\boldsymbol{0}}
\def\11{\mathds{1}}
\def\ggamma{\mbox{\boldmath{$\gamma$}}}

\def\DDelta{\mbox{\boldmath{$\Delta$}}}
\def\eepsilon{\mbox{\boldmath{$\epsilon$}}}

\def\xxi{\mbox{\boldmath{$\xi$}}}

\def\mmu{\mbox{\boldmath{$\mu$}}}

\def\zzeta{\mbox{\boldmath{$\zeta$}}}
\newcommand{\off}[1]{}

\newtheorem{theorem}{Theorem}
\newtheorem{lemma}[theorem]{Lemma}

\newenvironment{proof}{\noindent{\bf{Proof.\/}}}{\hfill$\blacksquare$\vskip0.1in}

\newcommand{\beq}{\begin{equation}}
\newcommand{\eeq}{\end{equation}}

\newcommand{\ben}{\begin{enumerate}}
\newcommand{\een}{\end{enumerate}}

\begin{document}
\title
{\bf  A Convergence Study for  Reduced Rank  Extrapolation  on Nonlinear Systems}
\author
{Avram Sidi\\
Computer Science Department\\
Technion - Israel Institute of Technology\\ Haifa 32000, Israel\\
e-mail:{~~~}\url{asidi@cs.technion.ac.il}\\
URL:{~~~}\url{http://www.cs.technion.ac.il/~asidi}}
\date{Appeared in: \ {\em Numerical Algorithms}, 84:957--982, 2020}
\maketitle
\vspace{-1cm}

\thispagestyle{empty}
\newpage

\begin{abstract}\sloppypar
Reduced Rank Extrapolation  (RRE) is a polynomial type  method used to accelerate the  convergence  of  sequences of vectors $\{\xx_m\}$. It is applied successfully in
 different disciplines  of science and engineering in
 the solution of large and sparse systems of linear and  nonlinear equations of very large dimension.
 If $\sss$ is the solution to the system of equations $\xx=\ff(\xx)$, first, a vector sequence $\{\xx_m\}$ is generated   via the fixed-point iterative scheme
 $\xx_{m+1}=\ff(\xx_m)$, $m=0,1,\ldots,$ and next,
 RRE is applied to this sequence to accelerate its convergence. RRE produces
 approximations $\sss_{n,k}$ to $\sss$ that are of the form $\sss_{n,k}=\sum^k_{i=0}\gamma_i\xx_{n+i}$  for some scalars  $\gamma_i$
 depending (nonlinearly) on $\xx_n, \xx_{n+1},\ldots,\xx_{n+k+1}$ and satisfying  $\sum^k_{i=0}\gamma_i=1$. The convergence properties of RRE when applied in conjunction with linear $\ff(\xx)$ have been analyzed in different publications.
 In this work, we discuss  the  convergence of the $\sss_{n,k}$ obtained from RRE with nonlinear $\ff(\xx)$  (i)\,when $n\to\infty$  with fixed  $k$, and (ii)\,in two  so-called {\em cycling} modes.

\end{abstract}
\vspace{1cm} \noindent {\bf Mathematics Subject Classification
2000:}  65B05 (primary), 65H10 (primary), 65F10 (secondary).

\vspace{1cm} \noindent {\bf Keywords and expressions:} Vector extrapolation methods,
minimal polynomial extrapolation (MPE), reduced rank extrapolation (RRE), Krylov subspace methods, nonlinear equations, cycling mode.

\thispagestyle{empty}
\newpage
\pagenumbering{arabic}

\section{Introduction} \label{se1}
\setcounter{equation}{0}
\setcounter{theorem}{0}
Consider a system of  nonlinear algebraic equations of dimension $N$, which we choose to write as
\beq\label{eq1} \xx=\ff(\xx), \quad \ff:\C^N\to\C^N;\quad \text{$\sss$\ solution,}\eeq
where
\beq \xx=[x^{(1)},\ldots,x^{(N)}]^{T},\quad \sss=[s^{(1)},\ldots,s^{(N)}]^{T};
\quad \label{eq1q}x^{(i)},\, s^{(i)}\quad\text{scalars,}\eeq and
\beq \label{eq1w}\ff(\xx)=\big[f_1(\xx),\ldots,f_N(\xx)\big]^{T};
\quad f_i(\xx)=f_i\big(x^{(1)},\ldots,x^{(N)}\big)\quad\text{scalar
functions.}\eeq
One immediate way of solving this system is via the fixed-point iterative scheme
\beq\label{eq2} \xx_{m+1}=\ff(\xx_{m}),\quad m=0,1,\ldots;\quad \text{for some $\xx_0$,}\eeq
provided the sequence $\{\xx_m\}$ converges.
Let $\ff(\xx)$ be twice continuously differentiable in a neighborhood of $\sss$, and let
 $\FF(\xx)$ be the Jacobian matrix of $\ff$ evaluated at $\xx$, that is,
\beq\label{eq:jacobian}\FF(\xx)=\begin{bmatrix}
f_{1,1}(\xx)&f_{1,2}(\xx)&\cdots&f_{1,N}(\xx)\\
f_{2,1}(\xx)&f_{2,2}(\xx)&\cdots&f_{2,N}(\xx)\\
\vdots& \vdots& & \vdots \\
f_{N,1}(\xx)&f_{N,2}(\xx)&\cdots&f_{N,N}(\xx)
\end{bmatrix};\quad f_{i,j}(\xx)=\frac{\partial f_i}{\partial x^{(j)}}(\xx) .\eeq
It is known that (see Ortega and Rheinboldt \cite{Ortega:1970:ISN}, for example) if $\rho(\FF(\xx))$, the spectral radius of $\FF(\xx)$, is such that $\rho(\FF(\sss))<1$ and if $\xx_0$ is sufficiently close to $\sss$, then the sequence $\{\xx_m\}$ converges to $\sss$. The closer $\rho(\FF(\sss))$ is to one, the slower is the convergence of $\{\xx_m\}$ to $\sss$; this is the case in most practical engineering applications.

 The convergence of $\{\xx_m\}$ to $\sss$ can be accelerated substantially by applying to it a  vector  extrapolation method. When applied to $\{\xx_m\}$, an extrapolation method  produces
 approximations $\sss_{n,k}$ to $\sss$ that are, either directly or indirectly, of the form
 \beq \label{eq3}\sss_{n,k}=\sum^k_{i=0}\gamma_i\,\xx_{n+i};\quad \text{$\gamma_i$ \ some scalars},\quad
 \sum^k_{i=0}\gamma_i=1,\eeq
 the $\gamma_i$ depending nonlinearly on the $\xx_m$ used in constructing $\sss_{n,k}$.
 Let $M$ be  the number of the $\xx_m$  needed to construct $\sss_{n,k}$. (Of course, $M$ is not necessarily the same for all vector extrapolation methods.)\footnote{It is clear that the integers $n$ and $k$ are chosen by the user and that $M$ is determined by $n$, $k$, and the extrapolation method being used.}

For the sake of completeness,  here we mention briefly those vector extrapolation methods that have been shown to be useful in applications.

\begin{enumerate}
\item  {\bf Polynomial type methods:} These are {\em minimal polynomial extrapolation  (MPE)}, {\em reduced rank extrapolation (RRE)}, {\em modified  minimal polynomial extrapolation  (MMPE)}, and the most recent {\em singular value decomposition-based minimal polynomial extrapolation (SVD-MPE)}. MPE was introduced by Cabay and Jackson \cite{Cabay:1976:PEM}, RRE was introduced independently by
Kaniel and Stein \cite{Kaniel:1974:LSA}, Eddy \cite{Eddy:1979:ELV}, and
Me\u{s}ina \cite{Mesina:1977:CAI}.\footnote{The approaches of \cite{Kaniel:1974:LSA} and \cite{Mesina:1977:CAI} to RRE are almost identical, in the sense that $\sss_{n,k}=\sum^k_{i=0}\gamma_i\,\xx_{n+i}$ in \cite{Mesina:1977:CAI}, while $\sss_{n,k}=\sum^k_{i=0}\gamma_i\,\xx_{n+i+1}$
in \cite{Kaniel:1974:LSA}, the $\gamma_i$ being the same for both.
The approaches of \cite{Eddy:1979:ELV} and \cite{Mesina:1977:CAI} are completely different, however; their equivalence was proved in the review paper of Smith, Ford, and Sidi \cite{Smith:1987:EMV}.}
  MMPE was introduced independently by
Brezinski \cite{Brezinski:1975:GTS}, Pugachev \cite{Pugachev:1978:ACI}, and
Sidi, Ford, and Smith \cite{Sidi:1986:ACV}.   SVD-MPE is a new method by Sidi
 \cite{Sidi:2016:SVD-MPE}.

\item {\bf Epsilon algorithms:} These are the {\em scalar epsilon algorithm (SEA)}, the {\em vector epsilon algorithm (VEA)}, and the
{\em topological epsilon algorithm (TEA)}. SEA is a method that is based entirely on the famous {\em epsilon algorithm} of Wynn \cite{Wynn:1956:DCT} that implements the transformation of  Shanks \cite{Shanks:1955:NTD} for scalar  sequences. VEA was introduced by  Wynn \cite{Wynn:1962:ATI}. TEA was introduced by Brezinski \cite{Brezinski:1975:GTS}.
\end{enumerate}

 For an earlier account of the epsilon algorithms,  see the book by Brezinski \cite{Brezinski:1977:ACA}. For a comprehensive survey covering the developments that took place until the 1980s, see the survey paper by Smith, Ford, and Sidi \cite{Smith:1987:EMV}
 and the book
  by Brezinski and Redivo Zaglia \cite{Brezinski:1991:EMT}.
   For a geometric approach to the treatment of vector extrapolation methods as these are being applied to linear systems, see Jbilou and Sadok \cite{Jbilou:1995:ASV}. For a more recent review of MPE and RRE, see Sidi \cite{Sidi:2012:RTV}. For a detailed and up-to-date treatment, including development, analysis,  numerical implementation, and various applications,  of all these methods, see the recent book of Sidi \cite{Sidi:2017:VEM}.

Numerically stable and efficient algorithms for  implementing polynomial methods have been proposed by Sidi \cite{Sidi:1991:EIM}, \cite{Sidi:2016:SVD-MPE} for  MPE, RRE, and SVD-MPE and by Jbilou and Sadok \cite{Jbilou:1999:LUI} for MMPE. The epsilon algorithms are normally implemented via their definitions, which involve recursion relations. When applied to sequences $\{\xx_m\}$ generated via  fixed-point iterative schemes from systems of  linear equations, MPE, RRE, and TEA turn out to be equivalent to known Krylov subspace methods for linear systems. This is explored in Sidi \cite{Sidi:1988:EVP}. Yet another recent  paper by Sidi \cite{Sidi:2017:MPERRE} shows that MPE and RRE are very closely related in more than one  way.

Now, all the methods mentioned above have interesting    convergence and convergence acceleration properties that concern the precise asymptotic behavior of the sequences $\{\sss_{n,k}\}^\infty_{n=0}$, with fixed $k$, when the sequences  $\{\xx_m\}$ are generated via  fixed-point iterative schemes from systems of  linear equations; see   Sidi \cite{Sidi:1986:CSP}, \cite{Sidi:1994:CIR}, Sidi, Ford, and Smith \cite{Sidi:1986:ACV}, and Sidi and Bridger \cite{Sidi:1988:CSA}, and also Sidi \cite[Chapter 6]{Sidi:2017:VEM} for the methods MPE, RRE, MMPE, and TEA, Wynn \cite{Wynn:1966:CSE} and
 Sidi  \cite{Sidi:1996:ECW} for SEA, and  Graves-Morris and Saff \cite{Graves:1988:RCT} for  VEA. We shall call this mode of usage of vector extrapolation methods the $n$-{\em Mode}.

 Unfortunately, the $n$-Mode convergence theories  that apply to the case in which $\ff(\xx)$ is linear do not apply to the case in which $\ff(\xx)$ is nonlinear. This is one of the topics we would like to study here, RRE being the extrapolation method used.
 That is,  we would like to investigate the convergence properties of the sequences $\{\sss_{n,k}\}^\infty_{n=0}$, with fixed $k$, obtained by applying RRE  to $\{\xx_m\}$ generated as in \eqref{eq2}, where $\ff(\xx)$ is nonlinear.

The  numerical implementations of  polynomial extrapolation methods  and of epsilon algorithms, when  generating the vectors $\sss_{n,k}$,   necessitate the keeping of resp. $k+2$ and $2k+1$ vectors in core memory simultaneously.  In case we would like to increase $k$ to improve the quality of the $\sss_{n,k}$, this may pose  a
serious  problem when  we are dealing with very high dimensional vectors, which is the case in most large scale applications.
Within the context described via \eqref{eq1}--\eqref{eq2} in the first paragraph of this section, it is best to apply
vector extrapolation methods in the so-called {\em cycling} mode, and this has been the usual practice.  This mode of usage of vector extrapolation methods, which we  shall call the {\em C-Mode}, can be described via the following steps:
\bigskip

\subsection*{\underline{\bf C-Mode}}
\begin{enumerate}
\item  [\textnormal{C0.}] Choose integers $n\geq 0$ and  $k\geq
1$ and an initial vector $\xx_0$.
\item [\textnormal{C1.}]
Compute the vectors $\xx_1,\xx_2,\ldots,\xx_M$ [via
$\xx_{m+1}=\ff(\xx_m)$].\footnote{Note that $M=n+k+1$ for MPE, RRE, MMPE, and SVD-MPE, while
$M=n+2k$ for SEA, VEA, and  TEA.  \label{f1}}
\item
[\textnormal{C2.}] Apply the   extrapolation method
 to the vectors
$\xx_n,\xx_{n+1},\ldots,\xx_{M}$, and compute $\sss_{n,k}$.
 \item [\textnormal{C3.}]
If $\sss_{n,k}$ satisfies the accuracy test, stop.\\
Otherwise, set $\xx_0=\sss_{n,k}$ and go to step C1.
\end{enumerate}
 We  call each application of steps C1--C3  a {\em cycle} and
denote by $\sss^{(r)}$ the $\sss_{n,k}$ computed in
the $r$th cycle. We will also denote the initial vector $\xx_0$ in
step C0 by $\sss^{(0)}$. Under suitable conditions, it has been shown rigorously
for MPE and RRE that the
sequence $\{\sss^{(r)}\}^\infty_{r=0}$ has very good
convergence properties when $\ff(\xx)$ is linear.
See  \cite{Sidi:1991:UBC}, \cite{Sidi:1998:UBC}. See also  \cite[Chapter 7]{Sidi:2017:VEM}.
The case in which $\ff(\xx)$ is nonlinear has proved to be  complicated and has not been  resolved till the present.

In some cases, RRE stalls if applied in the C-Mode  with $n=0$, in the sense  that it takes too many iterations until one sees meaningful convergence; in such cases, even a moderate $n>0$ can be very helpful to accelerate convergence effectively. See the numerical examples in \cite{Sidi:1991:UBC}, \cite{Sidi:1998:UBC}.
\bigskip

A different cycling procedure involving the minimal polynomial of the (constant) Jacobian matrix $\FF(\sss)$ with respect to a nonzero vector\footnote{Given a  nonzero vector $\uu\in\C^N,$ the monic polynomial
$P(\lambda)$ is said to be a {\em minimal polynomial of the matrix $\TT\in\C^{N\times N}$ with respect to $\uu$}
if $P(\TT)\uu=\00$ and if $P(\lambda)$ has
smallest degree.\\
The polynomial $P(\lambda)$ exists and is unique. Moreover, if $P_1(\TT)\uu = \00$ for some polynomial $P_1(\lambda)$ with $\deg P_1 > \deg P$, then $P(\lambda)$ divides $P_1(\lambda)$. In particular, $P(\lambda)$ divides
the minimal polynomial of $\TT$, which in turn divides the characteristic polynomial of $\TT$.
[Thus, the degree of $P(\lambda)$ is at most $N$ and its zeros are some or all of the eigenvalues of $\TT$.]}
  has been considered in various publications. The description of this procedure, which we  shall call the {\em MC-Mode}, is as follows:
\bigskip

\subsection*{\underline{\bf MC-Mode}}
\begin{enumerate}
\item  [\textnormal{MC0.}] Choose an integer $n\geq0$ and an initial vector $\xx_0$.
\item [\textnormal{MC1.}] Compute the vectors $\xx_1,\xx_2,\ldots,\xx_M$ [via
$\xx_{m+1}=\ff(\xx_m)$], $M$ being as explained in footnote$^{\ref{f1}}$, with $k$ there being the
degree of the minimal polynomial of $\FF(\sss)$  with respect to $\eepsilon_n=\xx_{n}-\sss$.\footnote{It is clear that to apply any of the extrapolation methods in this mode, one needs to know the matrix $\FF(\sss)$, for which one also needs to know the solution $\sss$.}
\item
[\textnormal{MC2.}]
Apply the   extrapolation method  to the vectors
$\xx_n,\xx_{n+1},\ldots,\xx_{M}$, and compute $\sss_{n,k}$.
 \item [\textnormal{MC3.}]
If $\sss_{n,k}$ satisfies the accuracy test, stop.\\
Otherwise, set $\xx_0=\sss_{n,k}$ and go to step MC1.
\end{enumerate}
As before, we  call each application of steps MC1--MC3  a {\em cycle} and
denote by $\sss^{(r)}$ the $\sss_{n,k}$ computed in
the $r$th cycle.\footnote{Note that $k$ is not necessarily fixed in this mode of cycling; it may vary from one cycle to the next. It always satisfies $k\leq N$, however.} We will also denote the initial vector $\xx_0$ in
step MC0 by $\sss^{(0)}$. It is observed  in many numerical examples that
 the
sequence $\{\sss^{(r)}\}^\infty_{r=0}$ converges quadratically to the solution $\sss$ of the system $\xx=\ff(\xx)$ when $\ff(\xx)$ is nonlinear.\footnote{\label{foot5}Quadratic convergence is relevant only when $\ff(\xx)$ is nonlinear. When $\ff(\xx)$ is linear, that is, $\ff(\xx)=\TT\xx+\dd$, where $\TT$ is a fixed $N\times N$ matrix and $\dd$ is a fixed vector, hence $\FF(\sss)=\TT$, the solution $\sss$ is obtained already at the end of step MC2 of  the first cycle, that is, we have $\sss^{(1)}=\sss$. Therefore, there is nothing to analyze  when $\ff(\xx)$ is linear.}
The first papers dealing with this topic (that is the MC-Mode with $\sss_{0,k}$ only) are those by Brezinski \cite{Brezinski:1970:AEA}, \cite{Brezinski:1971:ARS},   Gekeler \cite{Gekeler:1972:SSE}, and Skelboe \cite{Skelboe:1980:CPS}. Of these, \cite{Brezinski:1970:AEA}, \cite{Brezinski:1971:ARS},  and \cite{Gekeler:1972:SSE}  consider the application of the epsilon algorithms,  while \cite{Skelboe:1980:CPS} also considers the application of MPE and RRE.
The quadratic convergence proofs in all of these papers have a gap in that
they all end up with the relation
$$\|\sss^{(r+1)}-\sss\|_2\leq K_r\|\sss^{(r)}-\sss\|_2^2,$$
from which they  conclude that
$\{\sss^{(r)}\}^\infty_{r=0}$ converges quadratically. However, $K_r$ is a scalar that depends on $r$ through $\sss^{(r)}$, and the proofs do not show how it depends on $r$. In particular, they do not show whether $K_r$ is  bounded in $r$ or how it grows with $r$ if it is not bounded. This gap was disclosed in the review paper of Smith, Ford, and Sidi \cite{Smith:1987:EMV}.

A more recent paper by Jbilou and Sadok \cite{Jbilou:1991:SRA} deals with the same MC-Mode cycling via MPE and RRE. Yet another paper by Le Ferrand \cite{LeFerrand:1992:CTE} treats TEA.
Both these works provide  proofs of quadratic convergence by imposing  some global  conditions on the whole sequence $\{\sss^{(r)}\}^\infty_{r=0}$ as well as on $\ff(\xx)$.   (See also Laurens and Le Ferrand \cite{Laurens:1995:FIV}.)

In this work, we present a new  convergence
 study of RRE  when it is being applied to nonlinear systems. Specifically, we treat the convergence of RRE (i)\,in the $n$-Mode,  and (ii)\,in the two cycling modes mentioned above. By making a {\em global} assumption, we are able to prove convergence in all cases. We can justify heuristically the
 plausibility of this assumption;  we do not have a rigorous justification for it, however. This difficulty is  inherent to  all studies. We  explore the source of this difficulty here.
It must be mentioned that the difficulties that exist in the previous papers mentioned above are similar to  ours, although they take  different forms. Whether and how we can circumvent these difficulties  is not clear at this time.

The plan of this paper is as follows:
In Section \ref{se2}, we give a brief description of RRE, which is needed throughout.
In Section \ref{se3}, we derive a formula for the error vector $\sss_{n,k}-\sss$ when the vectors $\xx_m$ are generated via \eqref{eq2} with a nonlinear $\ff(\xx)$. In Section~\ref{se4},
we use this error formula  to derive an  upper bound on $\|\sss_{n,k}-\sss\|$. In Section~\ref{se5}, we complete the convergence studies of RRE in the different modes mentioned above. We mention that our results concerning the convergence of RRE in the $n$-Mode and the C-Mode are the first ones in the literature of extrapolation methods. In our study, we make much use of the results presented in Sidi \cite{Sidi:1988:EVP} throughout these studies. In Section~\ref{se6}, we discuss the nature of the problem/difficulty mentioned above and compare our global assumption with that  of \cite{Jbilou:1991:SRA}.
 In the appendix, we  review some
known theorems concerning Moore--Penrose generalized inverses of perturbed matrices, which we use in Section \ref{se4}.
(For generalized inverses, see  Ben-Israel and Greville \cite{Ben-Israel:2003:GIT} and Campbell and Meyer \cite{Campbell:1991:GIL}, for example.)

 Throughout this work,  we will use lowercase boldface italic letters to denote vectors and we will use uppercase  boldface italic letters to denote matrices.

Finally, we mention that in our study of RRE, we employ two different vector norms:
 \begin{description}
 \item [(i)] The standard $l_2$ vector norm defined
via $\|\zz\|_2=\sqrt{\zz^*\zz}$  and the matrix norm induced by it, namely,
$\|\AAA\|_2=\sigma_\text{max}(\AAA)$, where $\sigma_\text{max}(\AAA)$ is the largest singular value of the matrix $\AAA$.
\item [(ii)]
The $\GG$ norm defined via $\|\zz\|=\|\GG\zz\|_2$ and the  matrix norm $\|\AAA\|$ induced by it, where $\GG=\II-\FF(\sss)$. Note that $\GG$ is nonsingular since $\FF(\sss)$ does not have unity as an eigenvalue; therefore,
the $\GG$ norm is a true vector norm.
\end{description}
 Of course, the two vector norms are equivalent and we have
\beq\label{equiv}\frac{1}{\|\GG^{-1}\|_2}\|\zz\|_2\leq \|\zz\|\leq \|\GG\|_2\,\|\zz\|_2.\eeq
When $\AAA$ is an $N\times N$  (square) matrix,  we have
$\|\AAA\|=\|\GG\AAA\GG^{-1}\|_2$. We make extensive use of these connections between the two norms, $\|\cdot\|$ and $\|\cdot\|_2$, in the sequel.

\section{Description of RRE} \label{se2}
\setcounter{equation}{0}
\setcounter{theorem}{0}
Consider the system of equations given in \eqref{eq1}--\eqref{eq1w}, and let
 the sequence $\{\xx_m\}$ be  generated via the fixed-point iterative scheme in \eqref{eq2}.

Define the first and second order differences of the $\xx_m$ as in
\beq \label{eq12} \uu_m=\xx_{m+1}-\xx_m,\quad \ww_m=\uu_{m+1}-\uu_m=\xx_{m+2}-2\xx_{m+1}+\xx_m,\quad m=0,1,\ldots,\eeq
and, for some fixed $n\geq0$, form the $N\times (j+1)$ matrices
\beq \label{eq13} \UU_j=[\,\uu_n\,|\,\uu_{n+1}\,|\,\cdots\,|\,\uu_{n+j}\,],\quad \WW_j=[\,\ww_n\,|\,\ww_{n+1}\,|\,\cdots\,|\,\ww_{n+j}\,],\quad j=0,1,\ldots.\eeq
Then the  $\gamma_i$ in \eqref{eq3} for RRE are  the solution to the constrained standard $l_2$ minimization problem
\beq\label{eq14}
\min_{\gamma_0,\gamma_1,\ldots,\gamma_k}\bigg\|\sum^k_{i=0}\gamma_i\uu_{n+i}\bigg\|_2\quad \text{subject to}\ \sum^k_{i=0}\gamma_i=1,\eeq
which can also be expressed in matrix terms as

\beq \label{eq15}
\min_{\ggamma}\|\UU_k\ggamma\|_2\quad \text{subject to}\ \sum^k_{i=0}\gamma_i=1;\quad \ggamma=[\gamma_0,\gamma_1,\ldots,\gamma_k]^T\in\C^{k+1}. \eeq
Then, with the solution $\ggamma$ of this problem, the RRE approximation $\sss_{n,k}$ is given as in
\beq \label{eq16} \sss_{n,k}=\sum^k_{i=0}\gamma_i\,\xx_{n+i}.\eeq

Noting that
$$\xx_{n+m}=\xx_n+\sum^{m-1}_{j=0}\uu_{n+j}, \quad \uu_{n+m}=\uu_n+\sum^{m-1}_{j=0}\ww_{n+j},\quad m=0,1,\ldots, $$
we can  reexpress $\sss_{n,k}$ and $\UU_k\ggamma$  as
\beq\label{eq18} \sss_{n,k}=\xx_n+\sum^{k-1}_{j=0}\xi_j\,\uu_{n+j}=\xx_n+\UU_{k-1}\xxi,\quad \UU_k\ggamma=\uu_n+\sum^{k-1}_{j=0}\xi_j\ww_{n+j}=\uu_n+\WW_{k-1}\xxi, \eeq where
\beq\label{eq19} \xxi=[\xi_0,\xi_1,\ldots,\xi_{k-1}]^T\in\C^k;\quad \xi_j=\sum^k_{i=j+1}\gamma_i,\quad j=0,1,\ldots,k-1.\eeq
The (constrained)  minimization problem for the vector $\ggamma$ in \eqref{eq15}
can now be replaced by the following (unconstrained)  minimization problem for the vector $\xxi$  in \eqref{eq18}:
\beq \label{eq20} \min_{\xxi}\|\uu_n+\WW_{k-1}\xxi\|_2,\quad \xxi=[\xi_0,\xi_1,\ldots,\xi_{k-1}]^T\in\C^k.\eeq
Now, the solution to this problem (for $\xxi$) is simply $-\WW{}^+_{k-1}\uu_n$, where $\KK^+$ stands for the Moore--Penrose generalized inverse of the matrix $\KK$. Upon substituting this into \eqref{eq18}, we obtain

\beq\label{eq21}\boxed{ \sss_{n,k}=\xx_n-\UU_{k-1}\WW{}^+_{k-1}\uu_n.}\eeq
We will be making use of this representation of $\sss_{n,k}$ in the sequel.
For the above developments, see  Sidi \cite{Sidi:1986:CSP}.

\section{An error formula for RRE} \label{se3}
\setcounter{equation}{0}
\setcounter{theorem}{0}

\subsection{RRE on the linear system $\xx=\sss+{\FF}(\sss)(\xx-\sss)$} \label{sse31}
Let us now consider the linear system
\beq \label{eq40} \xx=\tilde{\ff}(\xx),\quad \tilde{\ff}(\xx)=\sss+{\FF}(\sss)(\xx-\sss),\eeq
where $\FF(\sss)$ is the Jacobian matrix of $\ff$ evaluated at $\sss$, as given in \eqref{eq:jacobian}.
Note that $\tilde{\ff}(\xx)$ is simply the linear part of the Taylor series of $\ff(\xx)$ in \eqref{eq1} about $\sss$.
 Clearly, $\sss$ is the solution to \eqref{eq40} since  $\tilde{\ff}(\sss)=\sss$.

 With the vectors  $\xx_0,\xx_1,\ldots,\xx_n$ generated nonlinearly as in \eqref{eq2} of  the preceding section, let
\beq \label{eq444}\tilde{\xx}_n=\xx_n\quad\text{and}\quad \tilde{\xx}_{m+1}=\tilde{\ff}(\tilde{\xx}_{m}),\quad m=n,n+1,\ldots.\eeq
Following this, define
\beq  \tilde{\eepsilon}_m=\tilde{\xx}_{m}-\sss,\quad
\tilde{\uu}_m=\tilde{\xx}_{m+1}-\tilde{\xx}_m,\quad \tilde{\ww}_m=\tilde{\uu}_{m+1}-\tilde{\uu}_m,\quad m=n,n+1,\ldots,\eeq
\beq \label{eq13f} \tilde{\UU}_j=[\,\tilde{\uu}_n\,|\,\tilde{\uu}_{n+1}\,|\,\cdots\,|\,\tilde{\uu}_{n+j}\,],\quad \tilde{\WW}_j=[\,\tilde{\ww}_n\,|\,\tilde{\ww}_{n+1}\,|\,\cdots\,|\,\tilde{\ww}_{n+j}\,],\quad j=0,1,\ldots.\eeq
Then, by \eqref{eq21},  the vector $\tilde{\sss}_{n,k}$ produced by applying RRE to the sequence $\{\tilde{\xx}_m\}$ is
\beq \label{eq81a} \boxed{ \tilde{\sss}_{n,k}=\tilde{\xx}_n-\tilde{\UU}_{k-1}\tilde{\WW}{}^+_{k-1}\tilde{\uu}_n.}\eeq
Upon subtracting $\sss$ from both sides of this equality and invoking $\tilde{\eepsilon}_n=\tilde{\xx}_n-{\sss}$, we obtain the error formula
\beq \label{eq81} \boxed{\tilde{\sss}_{n,k}-\sss=\tilde{\eepsilon}_n-\tilde{\UU}_{k-1}\tilde{\WW}{}^+_{k-1}\tilde{\uu}_n.}\eeq

 The error $\tilde{\sss}_{n,k}-\sss$   has been studied in detail in \cite{Sidi:1986:CSP},
 \cite{Sidi:1994:CIR},  \cite{Sidi:1988:CSA},
 \cite{Sidi:1991:UBC}; for a summary,
 see \cite[Chapters 6,7]{Sidi:2017:VEM}.\footnote{See also Sidi and Shapira \cite{Sidi:1998:UBC}
concerning a  modified version of restarted GMRES with prior Richardson iterations, that is very closely related to RRE.}

The following result from \cite[Theorem 4.2]{Sidi:1988:EVP} concerning the application of RRE to vector sequences from fixed-point iteration  of {\em linear} systems will be crucial in our analysis
of RRE concerning  {\em nonlinear} systems in Section \ref{se5}.

\begin{theorem} \label{thopt}
 Denote $\tilde{\FF}=\FF(\sss)$ for short; thus
 $\GG=\II-\tilde{\FF}$. Then the vector $\tilde{\sss}_{n,k}$ is the solution to the optimization problem
\begin{gather}  \|\tilde{\sss}_{n,k}-\sss\|=\|\GG(\tilde{\sss}_{n,k}-\sss)\|_2
=\min_{g\in\tilde{\cal P}_k}\|g(\tilde{\FF})\GG(\tilde{\xx}_n-\sss)\|_2, \notag \\
\tilde{\cal P}_k=\bigg\{g(z)=\sum^k_{j=0}\alpha_j z^j:\ \ g(1)=1\bigg\}. \label{eqopt} \end{gather}
\end{theorem}
\medskip

\noindent{\bf Remarks:}
\begin{enumerate}
\item
If $k$ is the degree of the minimal polynomial of ${\FF}(\sss)$ with respect to the vector $\tilde{\eepsilon}_n$, then $\tilde{\sss}_{n,k}=\sss$, the solution to \eqref{eq40}. See footnote$^{\ref{foot5}}$.

\item
Concerning Theorem \ref{thopt}, note that the vector $\GG(\yy-\sss)$ is simply the {\em residual}  of the vector $\yy$ for the linear system $\xx=\tilde{\ff}(\xx)$ because
$$\GG(\yy-\sss)=\yy-\tilde{\ff}(\yy),$$  and also
 $$\yy-\tilde{\ff}(\yy)=\00\quad \Leftrightarrow\quad
\yy=\sss, \quad \text{since $\GG$ is nonsingular.}$$
Thus, what Theorem \ref{thopt} means is that $\|\GG(\tilde{\sss}_{n,k}-\sss)\|_2$, the $l_2$ norm of the residual vector of  $\tilde{\sss}_{n,k}=\sum^k_{i=0}\tilde{\gamma}_i\tilde{\xx}_{n+i}$ subject to $\sum^k_{i=0}\tilde{\gamma}_i=1$, is the smallest of all the $l_2$ norms of the residuals of  the vectors
$\sum^k_{i=0}\alpha_i\tilde{\xx}_{n+i}$ subject to  $\sum^k_{i=0}\alpha_i=1$. Here we also recall that $\tilde{\xx}_n=\xx_n$ by \eqref{eq444}.
\end{enumerate}

\subsection{RRE on the nonlinear system $\xx=\ff(\xx)$}\label{sse32}

In the Introduction, we assumed  that $\ff(\xx)$ is twice continuously
 differentiable in a neighborhood of the solution $\sss$. We also assumed that $\rho(\tilde{\FF})<1$, where we recall $\tilde{\FF}= \FF(\sss)$, thus ensuring the convergence of the sequence $\{\xx_m\}$ to $\sss$. We now  assume, in addition, that
$\|\tilde{\FF}\|_2<1$ too and define  the ball $B(\sss,\delta)$  containing $\sss$ in its interior via
\beq\label{eq31} B(\sss,\delta)=\{\xx: \|\xx-\sss\|\equiv\|\GG(\xx-\sss)\|_2\leq \delta\}.\eeq
Clearly, $B(\sss,\delta)$ is a convex set. In addition, we assume $\ff(\xx)$ is twice continuously  differentiable in  $B(\sss,\delta)$.

\begin{lemma} \label{le32}For all $\delta$ sufficiently small, there exists a positive constant $L<1$ independent of $\delta$, such that
\beq \label{eq33} \|\xx_{m+1}-\sss\|\leq L \|\xx_{m}-\sss\|, \quad m=0,1,\ldots, \quad \text{provided\  $\xx_0\in B(\sss,\delta)$.}\eeq
Consequently, the whole sequence $\{\xx_m\}$ is in $B(\sss,\delta)$ and converges to $\sss$.
\end{lemma}
\begin{proof}
We begin with the following result that follows from Ortega and Rheinboldt \cite[p. 69]{Ortega:1970:ISN}:
$$ \ff(x)-\ff(\sss)=\int^1_0 \FF(\sss+t(\xx-\sss))(\xx-\sss)\,dt\quad \text{provided\ $\xx\in B(\sss,\delta)$}.$$
It is important to note that $\sss+t(\xx-\sss)$, with $t\in[0,1]$, is a convex combination of $\xx$ and $\sss$ hence is also in $B(\sss,\delta)$.
Multiplying both sides of this equality on the left by $\GG$, we obtain
$$ \GG[\ff(\xx)-\ff(\sss)]=\int^1_0 [\GG\FF(\sss+t(\xx-\sss))\GG^{-1}]\,[\GG(\xx-\sss)]\,dt,$$
which, upon taking $l_2$ norms on both sides and invoking the known fact that
$$ \bigg\|\int^b_a \uu(\xi)\,d\xi\bigg\|_2\leq \int^b_a\|\uu(\xi)\|_2\,d\xi,\quad \uu(\xi)\in \C^N,$$
 gives
$$\|\GG[\ff(\xx)-\ff(\sss)]\|_2\leq   \int^1_0 \|\GG\FF(\sss+t(\xx-\sss))\GG^{-1}\|_2\,\|\GG(\xx-\sss)\|_2\,dt. $$
Finally, invoking in this last inequality  $\|\GG\zz\|_2=\|\zz\|$ and the fact that $\|\GG\AAA\GG^{-1}\|_2=\|\AAA\|$, we obtain
\begin{align} \label{eq888} \|\ff(\xx)-\ff(\sss)\| &\leq \bigg(\int^1_0  \|\FF(\sss+t(\xx-\sss))\|\,dt\bigg)\|\xx-\sss\|\notag \\
&\leq \bigg[\max_{0\leq t\leq 1} \|\FF(\sss+t(\xx-\sss))\|\bigg]\,
\|\xx-\sss\|\notag \\
&\leq \bigg[\max_{\zz\in B(\sss,\delta)} \|\FF(\zz)\|\bigg]\,
\|\xx-\sss\|.\end{align}
Now, by the fact that $\ff(\xx)$ is twice differentiable in $B(\sss,\delta)$, it follows   that
\beq \label{eq887}\FF(\xx)=\FF(\sss+(\xx-\sss))=\FF(\sss)+\DDelta(\xx-\sss),\eeq where the matrix $\DDelta(\xx-\sss)$ satisfies
 \beq \label{eq886}\|\DDelta(\xx-\sss)\|\leq \alpha \|\xx-\sss\| \quad\text{for some $\alpha>0$ independent of $\xx\in B(\sss,\delta)$.}\eeq
Taking norms on  both sides of \eqref{eq887},  realizing that $\|\FF(\sss)\|=\|\FF(\sss)\|_2$ because $\FF(\sss)$ and $\GG=\II-\FF(\sss)$ commute, and invoking $\xx\in B(\sss,\delta)$, we have
\beq  \|\FF(\xx)\|=\|\GG\FF(\xx)\GG^{-1}\|_2\leq \|\FF(\sss)\|_2+\alpha\|\xx-\sss\|\leq
 \|\FF(\sss)\|_2+\alpha\delta\quad \forall \ \xx\in B(\sss,\delta).\eeq
Since we have assumed that $\|\FF(\sss)\|_2<1$, we can choose $\delta$  sufficiently small to cause
\beq \label{eq31k}\max_{\xx\in B(\sss,\delta)}\|\FF(\xx)\|=L<1.\eeq
With this, \eqref{eq888} becomes
\beq \label{eq32} \|\ff(\xx)-\ff(\sss)\|\leq L \|\xx-\sss\|\quad \forall\ \xx\in B(\sss,\delta).\eeq
The  proof of \eqref{eq33} for the sequence $\{\xx_m\}$ can now be carried out by letting $\xx=\xx_m$ in \eqref{eq32},  recalling that $\ff(\xx_m)=\xx_{m+1}$ and  $\ff(\sss)=\sss$, and then  proceeding by induction on $m$.
\end{proof}

In the sequel, we adopt the shorthand notation
\beq \label{eq30} \eepsilon_m=\xx_m-\sss,\quad m=0,1,\ldots;\quad \tilde{\FF}=\FF(\sss).\eeq
We also make use of the fact that $\|\tilde{\FF}\|\leq L<1$, which follows from \eqref{eq31k}, and, along with  \eqref{eq33}, guarantees that the sequence $\{\|\eepsilon_m\|\}$   decreases monotonically and converges to zero.

Expanding $\ff(\xx)$ in a Taylor series about the solution $\sss$ and using the fact that $\ff(\sss)=\sss$ and $\ff\in C^2(B(\sss,\delta))$, we have
 \beq \label{eq35} \ff(\xx)=\sss+\tilde{\FF}\cdot(\xx-\sss)+\mmu(\xx-\sss),\eeq
 where
\beq \label{eq35a}  \|\mmu(\xx-\sss)\|\leq a\,
\|\xx-\sss\|^2\quad \forall \, \xx\in B(\sss,\delta), \quad \text{for some $a>0$.}
\eeq
Consequently,
\beq \label{eq36} \xx_{m+1}=\ff(\xx_m)=\sss+\tilde{\FF}\eepsilon_m+\mmu(\eepsilon_m)\quad\Rightarrow\quad
\eepsilon_{m+1}=\tilde{\FF}\eepsilon_m+\mmu(\eepsilon_m).\eeq
Then, by induction,
\beq \label{eq37} \eepsilon_{n+i}=\tilde{\FF}{}^i\eepsilon_n+\sum^{i-1}_{j=0}\tilde{\FF}{}^{i-j-1}\mmu(\eepsilon_{n+j}),
\quad  i=0,1,2,\ldots.\eeq

\begin{lemma} \label{le1} The vectors $\eepsilon_m$, $\uu_m,$ and $\ww_m$ satisfy
\beq \label{eq38}\eepsilon_{n+i}=\tilde{\FF}{}^i\eepsilon_n+\check{\eepsilon}_{n+i};\quad \|\check{\eepsilon}_{n+i}\|\leq C_i\|\eepsilon_n\|^2,\quad C_i=a \frac{1-L^i}{1-L}L^{i-1},\eeq
\beq \label{eq38a}\uu_{n+i}=(\tilde{\FF}-\II)\tilde{\FF}{}^i\eepsilon_n+\check{\uu}_{n+i};\quad
\|\check{\uu}_{n+i}\|\leq D_i\|\eepsilon_n\|^2,\quad D_i=C_i+C_{i+1},\eeq
\beq
\label{eq39}\ww_{n+i}=(\tilde{\FF}-\II)^2\tilde{\FF}{}^i\eepsilon_n+\check{\ww}_{n+i};\quad
\|\check{\ww}_{n+i}\|\leq E_i\|\eepsilon_n\|^2,\quad E_i=C_i+2C_{i+1}+C_{i+2}.\eeq
\end{lemma}

\noindent{\bf Remark:} Note that $C_0=0$ and $C_1=a$ by \eqref{eq38}. Therefore,  $D_0=a$ by \eqref{eq38a}.\\

\begin{proof}
We start by noting that, by \eqref{eq37},
$$ \check{\eepsilon}_{n+i}=\sum^{i-1}_{j=0}\tilde{\FF}{}^{i-j-1}\mmu(\eepsilon_{n+j}),$$
which, upon taking norms and invoking $\|\tilde{\FF}\|\leq L$ and \eqref{eq35a}, gives
$$\|\check{\eepsilon}_{n+i}\|\leq \sum^{i-1}_{j=0}\|\tilde{\FF}{}^{i-j-1}\|\,\|\mmu(\eepsilon_{n+j})\|
\leq \sum^{i-1}_{j=0} L^{i-j-1}\,a\,(L^j\|\eepsilon_n\|)^2
=a\bigg(\sum^{i-1}_{j=0} L^{i+j-1}\bigg)\|\eepsilon_n\|^2,$$
from which \eqref{eq38} follows.

The proofs of \eqref{eq38a}--\eqref{eq39} follow  from \eqref{eq38} and the observation that
$$ \check{\uu}_m=\check{\eepsilon}_{m+1}-\check{\eepsilon}_m\quad\text{and}\quad  \check{\ww}_m=\check{\eepsilon}_{m+2}-2\check{\eepsilon}_{m+1}+\check{\eepsilon}_m.$$
We leave the details to the reader.
\end{proof}

Let us now go back to  the linear system $\xx=\tilde{\ff}(\xx)$ in \eqref{eq40},  recalling  that $\FF(\sss)=\tilde{\FF}$.
As already explained, $\tilde{\ff}(\xx)$ is simply the linear part of the Taylor series of $\ff(\xx)$ about $\sss$, obtained from \eqref{eq35} by letting $\mmu(\yy)\equiv\00$ there.  In addition, $\tilde{\ff}(\sss)=\sss$, that is, $\sss$ is the solution to $\xx=\tilde{\ff}(\xx)$, as well as $\xx={\ff}(\xx)$.
Let us now note  that $\mmu(\yy)\equiv\00$ also implies that $\check{\eepsilon}_m=\00$, $\check{\uu}_m=\00$, and $\check{\ww}_m=\00$
in \eqref{eq38}, \eqref{eq38a}, and \eqref{eq39}, respectively.  Recalling also that
$\tilde{\eepsilon}_n=\eepsilon_n$, we finally realize that, for $i=0,1,\ldots,$
\beq \label{eq41a}\tilde{\eepsilon}_{n+i}=\tilde{\FF}{}^i\tilde{\eepsilon}_n=
\tilde{\FF}{}^i\eepsilon_n, \quad \tilde{\uu}_{n+i}=\tilde{\FF}{}^i\tilde{\uu}_n=
(\tilde{\FF}-\II)\tilde{\FF}{}^i\eepsilon_n,\quad
\tilde{\ww}_{n+i}=\tilde{\FF}{}^i\tilde{\ww}_n=
(\tilde{\FF}-\II)^2\tilde{\FF}{}^i\eepsilon_n;
\eeq
consequently,
\beq\label{eq41h}\uu_{n+i}=\tilde{\uu}_{n+i}+\check{\uu}_{n+i}, \quad
\ww_{n+i}=\tilde{\ww}_{n+i}+\check{\ww}_{n+i}.\eeq
As a result of all this, we have
\beq\label{eq42} \UU_{k-1}=\tilde{\UU}_{k-1}+\check{\UU}_{k-1},\quad
\check{\UU}_{k-1}=[\,\check{\uu}_n\,|\,\check{\uu}_{n+1}\,|\,\cdots\,|\,\check{\uu}_{n+k-1}\,] \eeq
and
\beq \label{eq43} \WW_{k-1}=\tilde{\WW}_{k-1}+\check{\WW}_{k-1},\quad
\check{\WW}_{k-1}=[\,\check{\ww}_n\,|\,\check{\ww}_{n+1}\,|\,\cdots\,|\,\check{\ww}_{n+k-1}\,],\eeq
with $\UU_j$ and  $\WW_j$ as in \eqref{eq13}.
For  simplicity of notation, in what follows, we drop the subscript $k-1$ from the matrices $\UU_{k-1}$, $\WW_{k-1}$, $\tilde{\UU}_{k-1}$, $\tilde{\WW}_{k-1}$, etc.
 With these, \eqref{eq21} becomes
\begin{align} \sss_{n,k}&=\xx_n-\UU\WW{}^+\uu_n \notag\\
&=\xx_n-(\tilde{\UU}+\check{\UU})(\tilde{\WW}+\check{\WW})^+(\tilde{\uu}_n+\check{\uu}_n).\label{eq44} \end{align}
Letting also
\beq \label{eq48f} \HH=\WW{}^+-\tilde{\WW}{}^+=(\tilde{\WW}+\check{\WW}){}^+-\tilde{\WW}{}^+,\eeq
we   rewrite \eqref{eq44} in the form
\beq \label{eq48}
\sss_{n,k}=\xx_n-(\tilde{\UU}+\check{\UU})(\tilde{\WW}{}^++\HH)(\tilde{\uu}_n+\check{\uu}_n).\eeq
Next, opening the parentheses in \eqref{eq48}, we obtain the equality
\beq \label{eq82}
\sss_{n,k}=\xx_n-\tilde{\UU}\tilde{\WW}{}^+\tilde{\uu}_n
-\tilde{\UU}\tilde{\WW}{}^+\check{\uu}_n
-(\tilde{\UU}\HH+\check{\UU}\tilde{\WW}{}^++\check{\UU}\HH)(\tilde{\uu}_n+\check{\uu}_n).\eeq
 Now, by the fact that $\xx_n=\tilde{\xx}_n$ and by  \eqref{eq81a}, we have that
 $\xx_n-\tilde{\UU}\tilde{\WW}{}^+\tilde{\uu}_n=\tilde{\sss}_{n,k}$ in this equality. Next, we invoke
$\uu_n=\tilde{\uu}_n+\check{\uu}_n$ and $\UU=\tilde{\UU}+\check{\UU}$ again, and
 obtain a convenient representation of $\sss_{n,k}$ and the error in it. We summarize all this in the following lemma.

\begin{lemma}\label{le90} Let
\beq\label{scheck}
\check{\sss}_{n,k}=-\tilde{\UU}\tilde{\WW}{}^+\check{\uu}_n
-(\UU\HH+\check{\UU}\tilde{\WW}{}^+)\uu_n.\eeq
Then, $\sss_{n,k}$ is given by the equality
\beq \label{eq82c}
\sss_{n,k}= \tilde{\sss}_{n,k}+\check{\sss}_{n,k}.\eeq
Subtracting $\sss$ from both sides of this equality, we also obtain the error formula
\beq \label{eq87}
\sss_{n,k}-\sss=(\tilde{\sss}_{n,k}-\sss)+\check{\sss}_{n,k}.\eeq
\end{lemma}

\section{Derivation of  upper bounds for  $\|\sss_{n,k}-\sss\|$}\label{se4}
\setcounter{equation}{0}
\setcounter{theorem}{0}
\subsection{Preliminaries}
We now turn to the study of $\sss_{n,k}-\sss$. Multiplying both sides of \eqref{eq87} on the left by $\GG$ and taking $l_2$ norms, and also invoking $\|\zz\|_2\leq \|\GG^{-1}\|_2\,\|\zz\|$, we obtain

\begin{align}
&\|\sss_{n,k}-\sss\|\leq \|\tilde{\sss}_{n,k}-\sss\|+\|\check{\sss}_{n,k}\|, \notag\\
&\frac{\|\check{\sss}_{n,k}\|}{\|\GG^{-1}\|_2}\leq \|\GG\tilde{\UU}\|_2\,\|\tilde{\WW}{}^+\|_2\,\|\check{\uu}_n\|
+\|\GG\UU\|_2\,\|\HH\|_2\,\|\uu_n\|+\|\GG\check{\UU}\|_2\,\|\tilde{\WW}{}^+\|_2\,\|\uu_n\|. \label{eq87p}
\end{align}
Thus, we  need to study the behavior of each one of the terms in this bound.
We begin with the following lemma.
\begin{lemma}\label{le2} The following are true:
\begin{align}
 \|{\GG\UU}\|_2&\leq K_1\|\eepsilon_n\|,& \hspace{-1cm}
 \|\GG\tilde{\UU}\|_2 &\leq K_2\|\eepsilon_n\|,& \hspace{-1cm}
 \|\GG\check{\UU}\|_2&\leq K_3 \|\eepsilon_n\|^2, \label{eq76a} \\
 \|{\WW}\|_2&\leq K_1' \|\eepsilon_n\|,& \hspace{-1cm}
 \|\tilde{\WW}\|_2 &\leq K_2'\|\eepsilon_n\|,& \hspace{-1cm}
 \|\check{\WW}\|_2&\leq K_3' \|\eepsilon_n\|^2,\label{eq76b}
\end{align}
with $K_i$, $K_i'$, $i=1,2,3,$  positive constants independent of $k$ and $n$.
\end{lemma}
\begin{proof}
To achieve the proof, we make use of \eqref{eq38}--\eqref{eq41a} and
\beq \label{eq651} \begin{split}\|\uu_m\|\leq (1+L) \|\eepsilon_m\|  \quad \text{and}\quad\|\ww_m\|\leq (1+L)^2 \|\eepsilon_m\|, \\
\|\tilde{\uu}_m\|\leq (1+L) \|\eepsilon_m\|  \quad \text{and}\quad\|\tilde{\ww}_m\|\leq (1+L)^2 \|\eepsilon_m\|. \end{split}\eeq
We prove the validity of the bound on $\|\GG\UU\|_2$ only;  the others can be proved in exactly the same way.

We start by  analyzing $\|\GG\UU\|_F$, the Frobenius norm of $\GG\UU$.
Noting that
$$ \GG\UU=[\,\GG\uu_n\,|\,\GG\uu_{n+1}\,|\,\cdots\,|\,\GG\uu_{n+k-1}\,], $$
we have
\begin{align*} \|\GG\UU\|_F^2=\sum^{k-1}_{j=0}\|\GG\uu_{n+j}\|_2^2=\sum^{k-1}_{j=0}\|\uu_{n+j}\|^2 & \leq \sum^{k-1}_{j=0}[(1+L)\|\eepsilon_{n+j}\|]^2  \quad \text{by \eqref{eq651}} \\
& \leq (1+L)^2\sum^{k-1}_{j=0}(L^j\|\eepsilon_n\|)^2\quad \text{by \eqref{eq33}} \\
& =\frac{1+L}{1-L}(1-L^{2k})\|\eepsilon_n\|^2 \\
&< \frac{1+L}{1-L}\|\eepsilon_n\|^2.\end{align*}
The result $\|\GG\UU\|_2 \leq K_1\|\eepsilon_n\|$, with $K_1=\sqrt{(1+L)/(1-L)}$,  now follows by invoking $\|\GG\UU\|_2\leq\|\GG\UU\|_F$.\footnote{Recall that, for any matrix $\KK$ with $\text{rank}(\KK)=r$,
we have $\|\KK\|_2\leq \|\KK\|_F\leq r\|\KK\|_2$. See Golub and Van Loan \cite{Golub:2013:MC}.}
\end{proof}

 \subsection{Upper bounds for $\|\tilde{\WW}{}^+\|_2$ and $\|\HH\|_2$}
Next,  by Theorem \ref{thA2} in the appendix, we can bound $\|\HH\|_2$ as in
\beq \label{eq48g}\|\HH\|_2\leq \sqrt{2}\frac{\Delta}{1-\Delta}\|\tilde{\WW}{}^+\|_2\quad\text{provided \
$\Delta=\|\tilde{\WW}{}^+\|_2\,\|\check{\WW}\|_2<1$}.\eeq
We realize that all we need is a suitable  upper bound on $\|\tilde{\WW}{}^+\|_2$
since we already have an upper bound on $\|\check{\WW}\|_2$ from \eqref{eq663}.
We turn to this issue next.

\sloppypar
Now, by \eqref{eq13f} and \eqref{eq41a}, we have
\beq   \tilde{\WW}=[\,(\tilde{\FF}-\II)^2\eepsilon_n\,|\,(\tilde{\FF}-\II)^2\tilde{\FF}\eepsilon_n\,|
\,\cdots\,|\,(\tilde{\FF}-\II)^2\tilde{\FF}{}^{k-1}\eepsilon_n\,],\eeq
which can be written in the form
\beq \tilde{\WW}=\|\eepsilon_n\|_2\,\overset{\circ}{\WW},\quad \overset{\circ}{\WW}=\RR\SSS(\ee_n),\eeq where
\beq \RR=(\tilde{\FF}-\II)^2=\GG^2,\ \ \SSS(\yy)= [\,\yy\,|\,\tilde{\FF}\yy\,|
\,\cdots\,|\,\tilde{\FF}{}^{k-1}\yy\,],\quad \ee_n=\frac{\eepsilon_n}{\|\eepsilon_n\|_2}. \eeq
[Note that the columns of $\SSS(\yy)$ span the Krylov subspace
${\cal K}_k(\tilde{\FF};\yy)=\text{span}\{\yy,\tilde{\FF}\yy,\ldots,\tilde{\FF}{}^{k-1}\yy\}$.]
First, $\RR$ is $N\times N$, constant, and nonsingular since $\GG$ is.
Next, we recall that $k$ is at most the degree of the minimal polynomial of $\tilde{\FF}$ with respect to the vector $\eepsilon_n$, which implies that the vectors
$\tilde{\FF}{}^{i}\eepsilon_n$, $i=0,1,\ldots,k-1,$ are linearly independent and, therefore,
$\text{rank} (\SSS(\ee_n))=k$. As a result,
$\text{rank}(\tilde{\WW})=k=\text{rank}(\overset{\circ}{\WW})$ since $\RR$ is nonsingular.
 By  the fact that $(a \KK)^+=a^{-1}\KK^+$ for every nonzero scalar $a\in\C$, and by Theorem \ref{thAA} in the appendix, we thus have
\beq \label{eqrr1} \tilde{\WW}{}^+=\frac{1}{\|\eepsilon_n\|_2}\overset{\circ}{\WW}{}^+\quad\Rightarrow\quad
\|\tilde{\WW}{}^+\|_2=\frac{1}{\|\eepsilon_n\|_2}\|\overset{\circ}{\WW}{}^+\|_2 \eeq
and
\beq \label{eqrr2} \|\overset{\circ}{\WW}{}^+\|_2\leq \|\RR^{-1}\|_2\,\|\SSS(\ee_n)^+\|_2.\eeq
We need to bound only $\|\SSS(\ee_n)^+\|_2$ {\em uniformly} (i)\,for all $n=1,2, \ldots$ in the $n$-Mode, and (ii)\,for all unit vectors $\ee^{(r)}_n=\eepsilon^{(r)}_n/\|\eepsilon^{(r)}_n\|_2$ arising in the different cycles of the C-Mode and the MC-Mode. Unfortunately, we are not able to prove the existence of such uniform bounds.
In what follows, concerning the application  of RRE in the $n$-Mode and in two cycling modes,  we {\em assume} that, at each step of the different modes of usage of RRE,
$\|\SSS(\ee_n)^+\|_2$ is bounded uniformly throughout, that is, we assume that,
for some constant $\tilde{\eta}>0$,
\beq \label{eq662y}\|\SSS(\ee_n)^+\|_2\leq \tilde{\eta}. \eeq
Combining \eqref{eqrr1}--\eqref{eq662y} and invoking also $\|\eepsilon_n\|\leq \|\GG\|_2\,\|\eepsilon_n\|_2$, we obtain
\beq \label{eq666}\boxed{\|\tilde{\WW}{}^+\|_2\leq\frac{\eta}{\|\eepsilon_n\|},\quad \eta=\tilde{\eta}\|\GG\|_2\,\|\RR^{-1}\|_2}.\eeq
We shall  comment on this assumption concerning the uniform upper bound for $\|\SSS(\ee_n)^+\|_2$
in Section \ref{se6}.

The first thing to do now is to guarantee that $\Delta=\|\tilde{\WW}{}^+\|_2\,\|\check{\WW}\|_2<1$ in  \eqref{eq48g} is satisfied under the assumption in \eqref{eq666} concerning $\|\tilde{\WW}{}^+\|_2$. By \eqref{eq666} and \eqref{eq76b} and the fact that $\|\eepsilon_0\|\leq \delta$ since $\xx_0\in B(\sss,\delta)$, we have
\beq\label{eq663} \Delta\leq K_3'\eta\|\eepsilon_n\|\leq K_3'\eta L^n\|\eepsilon_0\|\leq  K_3'\eta L^n\delta.\eeq
Clearly, by making $\delta$ sufficiently small, we can make the upper bound on $\Delta$ smaller than one. The closer   $\delta$  is to zero, the closer $\xx_0$ is to $\sss$. This is precisely what is needed in order to develop a {\em local} convergence theory  for any extrapolation method.

Next, by \eqref{eq48g},  \eqref{eq666}, and \eqref{eq663},
\beq \label{eq663h}\|\HH\|_2\leq \lambda_n, \quad
\lambda_n=\sqrt{2}\frac{K_3'\eta^2}{1-K_3'\eta \|\eepsilon_n\|}.\eeq
As we will show later, $\eepsilon_n$ is bounded in all three modes ($n$-Mode, C-Mode, and MC-Mode) we study here, which   implies that $\lambda_n$ is bounded too.
\medskip

\noindent{\bf Remark:} Before proceeding further, we would like to discuss an interesting consequence of
the global assumption we have made concerning $\tilde{\WW}{}^+$.  By \eqref{eq666} and \eqref{eq663h}
and also by \eqref{eq48f}, namely, that $\WW{}^+=\tilde{\WW}{}^++\HH$, we have
$$ \|\WW{}^+\|_2\leq \|\tilde{\WW}{}^+\|_2+\|\HH\|_2\leq \frac{\eta}{\|\eepsilon_n\|}+\lambda_n.$$
As a result, the vector $\xxi=-\WW{}^+\uu_n$ defined via \eqref{eq20}, satisfies
$$ \|\xxi\|_2\leq \|\WW{}^+\|_2\, \|\uu_n\|_2\leq (1+L)(\eta+\lambda_n\|\GG^{-1}\|_2\,\|\eepsilon_n\|).$$
Here we have made use of \eqref{eq651} too.
Since $\|\eepsilon_n\|$ and $\lambda_n$ are bounded,  so is $\lambda_n\|\eepsilon_n\|$, in all three modes. This implies that $\xxi$ is bounded, which causes $\ggamma$ in \eqref{eq14}--\eqref{eq16} to be  bounded as well. This can be seen by expressing the $\gamma_i$ in terms of the $\xi_i$ by employing \eqref{eq19} as in
$$ \gamma_0=1-\xi_0;\quad \gamma_i=\xi_{i-1}-\xi_i,\quad i=1,\ldots,k-1;\quad \gamma_k=\xi_{k-1}.$$
Thus, we have globally
$$\boxed{\sum^k_{i=0}|\gamma_i|\leq\Gamma \quad \text{for some $\Gamma>0$ throughout all three modes.}}$$
Interestingly, this is analogous to the global assumption made by Toth and Kelly \cite{Toth:2015:CAA} in the convergence analysis of the  acceleration method of Anderson
\cite{Anderson:1965:IPN}. Note that, when applied to linear systems,
Anderson acceleration is equivalent to GMRES (see
Walker and Ni \cite{Walker:2011:AAF}), which is equivalent to RRE applied to linear systems
(see Sidi \cite{Sidi:1988:EVP}).

\subsection{Upper bound for $\|\sss_{n,k}-\sss\|$}
With the different matrices in \eqref{eq87p} bounded as above, we turn to $\sss_{n,k}-\sss$.
By \eqref{eq76a}, \eqref{eq76b}, \eqref{eq651}, and \eqref{eq666}, we have
\beq \|\GG\tilde{\UU}\|_2\,\|\tilde{\WW}{}^+\|_2\,\|\check{\uu}_n\|\leq
K_2\eta D_0 \|\eepsilon_n\|^2,\eeq
\beq \|\GG\UU\|_2\,\|\HH\|_2\,\|\uu_n\|\leq K_1\lambda_n(1+L)\,\|\eepsilon_n\|^2,\eeq
\beq \|\GG\check{\UU}\|_2\,\|\tilde{\WW}{}^+\|_2\,\|\uu_n\|\leq
K_3\eta(1+L)\,\|\eepsilon_n\|^2.\\
\eeq
Substituting  these into \eqref{eq87p}, we obtain
\beq \label{eq87u} \|\check{\sss}_{n,k}\|\leq \tau_n\,\|\eepsilon_n\|^2,\quad \tau_n=
[K_2\eta D_0+(K_1\lambda_n+K_3\eta)(1+L)]\,\|\GG^{-1}\|_2,\eeq
and this leads to the  bound on $\|\sss_{n,k}-\sss\|$ in the next lemma:
\begin{lemma}\label{le91}
The norm of the error vector $\sss_{n,k}-\sss$ can be bounded as in
\beq\label{eq87q}
\|\sss_{n,k}-\sss\|\leq \|\tilde{\sss}_{n,k}-\sss\|+\tau_n\,\|\eepsilon_n\|^2,\quad \tau_n=
[K_2\eta D_0+(K_1\lambda_n+K_3\eta)(1+L)]\,\|\GG^{-1}\|_2.\eeq
\end{lemma}

\noindent{\bf Remark:} Note that, by \eqref{eq663h} and \eqref{eq87q},  $\lim_{n\to\infty}\tau_n$ is finite since $\lim_{n\to\infty}\lambda_n$ is finite.
Therefore, $\|\sss_{n,k}-\sss\|$ cannot be smaller than $\|\check{\sss}_{n,k}\|\leq \tau_n\|\eepsilon_n\|^2$, even though $\|\tilde{\sss}_{n,k}-\sss\|$ may be smaller.
In other words, the term $\|\check{\sss}_{n,k}\|$ limits the accuracy of $\sss_{n,k}$
as an approximation to $\sss$.
\section{Convergence analysis}\label{se5}
\setcounter{equation}{0}
\setcounter{theorem}{0}
\subsection{Preliminaries}
We start by studying the term $\|\tilde{\sss}_{n,k}-\sss\|$. We recall that $\tilde{\sss}_{n,k}$ is the vector obtained by applying RRE to the
vectors $\tilde{\xx}_m$, $m=n,n+1,\ldots,n+k,$  with $\tilde{\xx}_n=\xx_n$, as described  in subsection \ref{sse31}.  Our study will be based on the developments of  \cite{Sidi:1988:EVP},   \cite{Sidi:1991:UBC}, and \cite[Chapters 6,7]{Sidi:2017:VEM}.

We first have
\begin{align} \label{eq987} \|\tilde{\sss}_{n,k}-\sss\|=\|\GG(\tilde{\sss}_{n,k}-\sss)\|_2&=\min_{g\in\tilde{\cal P}_k}
\|g(\tilde{\FF})\GG(\tilde{\xx}_n-\sss)\|_2 \quad \text{by  Theorem \ref{thopt}} \notag\\
&= \min_{g\in\tilde{\cal P}_k}
\|g(\tilde{\FF})\GG({\xx}_n-\sss)\|_2\quad \text{because $\tilde{\xx}_n=\xx_n$} \notag \\
&\leq \bigg[\min_{g\in\tilde{\cal P}_k}
\|g(\tilde{\FF})\|_2\bigg] \|\GG(\xx_n-\sss)\|_2\notag \\
& = \theta_k\|\eepsilon_n\|, \end{align}
 recalling that $\xx_n-\sss=\eepsilon_n$ and  defining
\beq \theta_k= \min_{g\in\tilde{\cal P}_k}\|g(\tilde{\FF})\|_2.\eeq
(Note that $\theta_k$ depends only on $\tilde{\FF}$ and $k$.)
Of course, we also have
\beq \label{eq987a}\theta_k\leq \|g(\tilde{\FF})\|_2\,\quad \forall\,  g\in\tilde{\cal P}_k.  \eeq
We now would like to bound $\theta_k$ appropriately.
Choosing  $g(z)=z^k$ in  \eqref{eq987a}, we obtain,
\beq\label{eqtk}\theta_k=\min_{g\in\tilde{\cal P}_k}\|g(\tilde{\FF})\|_2\leq \|\tilde{\FF}{}^k\|_2
\leq \|\tilde{\FF}\|_2^k\leq L^k,
\quad\text{at worst.}\footnote{Clearly, $g(z)=z^k$ is in $\tilde{\cal P}_k$ and $\theta_k<1$ since $L<1$. Next, in general, the polynomial $g(z)$ that gives the optimum in \eqref{eqtk} is different from $z^k$. Thus, generally speaking, $\theta_k<L^k$.}\eeq
With all these developments, \eqref{eq987} and  \eqref{eq87u} together give
 the result in the next lemma:
 \begin{lemma}\label{le78} The  error vector $\sss_{n,k}-\sss$ satisfies
\beq \label{eq135}\|{\sss}_{n,k}-\sss\|\leq  \theta_k\|\eepsilon_n\|+\tau_n
\|\eepsilon_n\|^2.\eeq
\end{lemma}

\noindent{\bf Remark:} By choosing $g(z)\in\tilde{\cal P}_k$ suitably,
upper bounds on $\theta_k$ that are
 smaller than $L^k$ can be given for different cases.
We give such bounds for two such cases here. For additional cases involving orthogonal polynomials, such as Jacobi polynomials, we refer the reader to Sidi and Shapira \cite{Sidi:1991:UBC}.
\begin{itemize}
\item
If  the hermitian part of $\GG=\II-\tilde{\FF}$, namely, the matrix $\GG_H=\frac{1}{2}(\GG+\GG^*)$, is positive definite, then
$$\theta_k\leq (1-\nu^2/\sigma^2)^{k/2},$$
where $\sigma$ is the largest singular value of $\GG$ and $\nu$ is the smallest eigenvalue of $\GG_H$.
Of course, $0<\nu<\sigma$. See \cite{Sidi:1988:EVP}.
\item
If $\tilde{\FF}$ is hermitian with eigenvalues in the (real) interval $[\alpha, \beta]$, $-1<\alpha<\beta<1$, then
$$ \theta_k\leq \frac{1}{T_k\big(\frac{2-\alpha-\beta}{\beta-\alpha}\big)}<2\bigg(\frac{\sqrt{\kappa}-1}
{\sqrt{\kappa}+1}\bigg)^k, \quad \kappa=\frac{1-\alpha}{1-\beta}>1.$$
Here $T_k(z)$ is the Chebyshev polynomial of the first kind of degree $k$.
(See Varga \cite[Chapter 5]{Varga:2000:MIA}, for example.)
Note that, in this case, $\theta_k< L^k$, with $L=\max(|\alpha|,|\beta|)<1.$
 \end{itemize}

\subsubsection*{Main assumptions}
Before delving into the local convergence analyses of the different modes of usage of RRE, we would like to summarize the assumptions we have made so far. We will be referring to them in the statements of our (local) convergence theorems below.
\begin{enumerate}
\item [A1.]
 $\ff\in C^2(B(\sss,\delta))$ for some $\delta>0$. (We can assume  $\delta$ to be as small as needed in our proofs.)
 \item [A2.]
$\|\tilde{\FF}\|\leq \max_{\xx\in B(\sss,\delta)}\|\FF(\xx)\|=L<1$, which also implies that $\rho(\tilde{\FF})\leq L.$
\item [A3.]
The very first vector $\xx_0$, with which we start any of the modes, is in $B(\sss,\delta) $. Thus,
 $\|\xx_0-\sss\|<\delta$.
\item [A4.]
$\|\tilde{\WW}{}^+\|\leq{\eta}/{\|\epsilon_n\|}$ \ \  for every $n$ in the $n$-Mode and for every cycle in the  C-Mode and the MC-Mode. ($\eta>0$ is fixed.)
\end{enumerate}

\subsection{Convergence in  $n$-Mode}\label{sse51}

We recall that, in the $n$-Mode, we are applying RRE, with  $k\ge1$ fixed throughout, to the {\em infinite} sequence $\{\xx_m\}$ that is generated as in \eqref{eq2}. (That is, no cycling is involved.)
\begin{theorem}
Under the assumptions A1--A4, RRE converges in the $n$-Mode. Actually, we have
\beq \label{eq654q} \limsup_{n\to\infty}\frac{\|{\sss}_{n,k}-\sss\|}{\|\eepsilon_n\|}\leq \theta_k<1.\eeq
\end{theorem}
\begin{proof}
 Since $\lim_{n\to\infty}\|\eepsilon_n\|=0$ and   $\lim_{n\to\infty}\tau_n<\infty$,    it is clear from Lemma \ref{le78} that $\lim_{n\to\infty}\|\sss_{n,k}-\sss\|=0,$ hence   $\lim_{n\to\infty}\sss_{n,k}=\sss$.

Next, again by Lemma \ref{le78}, we have

$$\frac{\|{\sss}_{n,k}-\sss\|}{\|\eepsilon_n\|}\leq \theta_k+
\tau_n\|\eepsilon_n\|.$$ Taking the limsup as $n\to\infty$ of both sides and recalling again that   $\lim_{n\to\infty}\tau_n<\infty$ and $\lim_{n\to\infty}\|\eepsilon_n\|=0$, the result in \eqref{eq654q}  follows.
\end{proof}

\noindent{\bf Remark:}
Let us also rewrite \eqref{eq87q} as
\beq \|\sss_{n,k}-\sss\|=O(\psi_n)\quad\text{as $n\to\infty$};\quad\psi_n=\max\{\|\tilde{\sss}_{n,k}-\sss\|, \|\eepsilon_n\|^2\}.\eeq
This is possible since $\lim_{n\to\infty}\tau_n$ is finite.
It is thus clear that $\|\sss_{n,k}-\sss\|$ cannot be less than $O(\|\eepsilon_n\|^2)=O(L^{2n})\approx O(\rho(\tilde{\FF})^{2n})$ as $n\to\infty$,
no matter what $\|\tilde{\sss}_{n,k}-\sss\|$ is. [See the remark following \eqref{eq87q}.]

\subsection{Convergence  in C-Mode cycling}\label{sse52}

 In C-Mode cycling, we keep $n\ge0$ and $k\ge1$ fixed throughout, $k$ always being assumed to be less than the degree of the minimal polynomial of $\tilde{\FF}$ with respect to the vector $\eepsilon_n$ in every cycle.
 \begin{theorem} \label{th53}
Under the assumptions A1--A4, RRE converges linearly in the C-Mode. Actually, we have
\beq \label{eq654} \limsup_{r\to\infty}\frac{\|{\sss}^{(r+1)}_{n,k}-\sss\|}
{\|{\sss}^{(r)}_{n,k}-\sss\|}\leq \theta_k L^n<1.\eeq
\end{theorem}
\begin{proof}
We start by observing that, by Lemma \ref{le32}, there holds $\|\eepsilon_n\|\leq L^n\|\eepsilon_0\|.$ With this, \eqref{eq135} becomes
\beq \label{eq012} \|{\sss}_{n,k}-\sss\|\leq \big(\theta_kL^n+\tau_n L^{2n}\|\eepsilon_0\|\big) \|\eepsilon_0\|.\eeq

Let us now denote the vectors $\xx_m$, $\eepsilon_m=\xx_m-\sss$, and $\sss_{n,k}$ used/computed in  cycle $i$ by $\xx_m^{(i)}$, $\eepsilon_m^{(i)}$, and $\sss_{n,k}^{(i)}$,
respectively,   and  rewrite \eqref{eq012} that is relevant to the  cycle $(r+1)$ as
\beq \label{eq012r}
\|\sss_{n,k}^{(r+1)}-\sss\|\leq \big(\theta_kL^n+\tau_nL^{2n}\|\eepsilon_0^{(r+1)}\|\big)\,\|\eepsilon_0^{(r+1)}\|.\eeq
Let us also  recall that, in the C-Mode,   $\xx^{(r+1)}_0=\sss^{(r)}_{n,k}$, and hence
 $\eepsilon_0^{(r+1)}=\sss^{(r)}_{n,k}-\sss$. With these,  \eqref{eq012r} becomes
\beq \label{eq012s}
\|\sss_{n,k}^{(r+1)}-\sss\|\leq \mu_r
\|\sss^{(r)}_{n,k}-\sss\|, \quad  \mu_r=\theta_{k}L^n+\tau_nL^{2n}\|\sss^{(r)}_{n,k}-\sss\|.
  \eeq
We now show by  induction that for each  $r$, $\sss_{n,k}^{(r)}$ is in the ball $B(\sss,\delta)$ and tends to $\sss$ as $r\to\infty$, provided $\xx_0$ in step C0 of C-Mode cycling is sufficiently close to $\sss$.

For $r=0$, we have $\xx_0=\sss^{(0)}_{n,k}\in B(\sss,\delta)$ by choice;  therefore,
$$ \mu_0\leq \theta_{k}L^n+\tau_nL^{2n}\delta.$$ Since $\theta_{k}L^n<1$, we can force $\mu_0<1$ by choosing $\delta$ sufficiently small or by choosing $\sss^{(0)}_{n,k}$ sufficiently close to $\sss$.  This, in turn,  forces
$$\|\sss^{(1)}_{n,k}-\sss\|\leq\mu_0\|\sss^{(0)}_{n,k}-\sss\|\leq\mu_0\delta<\delta\quad
\Rightarrow\quad \sss^{(1)}_{n,k}\in B(\sss,\delta).$$
 Continuing by induction on $r$, we
see that $\mu_{r+1}<\mu_{r}$,  $\|\sss^{(r+1)}_{n,k}-\sss\|<\|\sss^{(r)}_{n,k}-\sss\|$, hence
$\sss^{(r+1)}_{n,k}\in B(\sss,\delta)$ since $\sss^{(r)}_{n,k}\in B(\sss,\delta)$.
We also have $\lim_{r\to\infty}\|\sss_{n,k}^{(r)}-\sss\|=0,$  hence  $\lim_{r\to\infty}\sss_{n,k}^{(r)}=\sss$.
With the convergence of $\{\sss_{n,k}^{(r)}\}^\infty_{r=0}$ to $\sss$ established, let us now rewrite \eqref{eq012s} as
\beq\frac{\|\sss_{n,k}^{(r+1)}-\sss\|}
{\|\sss_{n,k}^{(r)}-\sss\|}\leq\mu_r.\eeq
Taking the limsup as $r\to\infty$ on both sides of this inequality, we obtain \eqref{eq654}.
\end{proof}

\subsection{Convergence in MC-Mode cycling}
We recall that in MC-Mode cycling, we keep $n$ fixed while $k=k_r$ is the  degree of the minimal polynomial of $\tilde{\FF}$ with respect to $\eepsilon_n$ in the $r$th cycle.

 \begin{theorem}
Under the assumptions A1--A4, RRE converges quadratically in the MC-Mode. Actually, we have
\beq \label{eq654d}\limsup_{r\to\infty}\frac{\|\sss_{n,k_{r+1}}^{(r+1)}-\sss\|}
{\|\sss_{n,k_r}^{(r)}-\sss\|^2}\leq\tau_nL^{2n}.\eeq
\end{theorem}

\begin{proof}
We start by noting that $\tilde{\sss}_{n,k_r}=\sss$ in each  cycle, as mentioned  at the end of subsection \ref{sse31}.
  Thus, \eqref{eq87q} becomes
\beq \label{eq375} \|{\sss}_{n,k}-\sss\|\leq \tau_n\|\eepsilon_n\|^2.\eeq
Proceeding precisely as in the proof of Theorem \ref{th53} concerning the C-Mode cycling,   we next obtain

\beq \label{eq375z} \|{\sss}_{n,k}-\sss\|\leq \tau_nL^{2n}\|\eepsilon_0\|^2.\eeq
As in the case of the C-Mode,  noting that  $\eepsilon_0^{(r+1)}=\sss^{(r)}_{n,k_r}-\sss$, we write \eqref{eq375z} in the MC-Mode as
\beq \label{eq890}\|{\sss}_{n,k_{r+1}}^{(r+1)}-\sss\|\leq \tau_nL^{2n}
\|{\sss}_{n,k_r}^{(r)}-\sss\|^2=\phi_r\,\|{\sss}_{n,k_r}^{(r)}-\sss\|,
\quad \phi_r=
\tau_nL^{2n}\|{\sss}_{n,k_r}^{(r)}-\sss\|.\eeq
We now show, by  induction on $r$, that $\sss_{n,k_r}^{(r)}$ is in the ball $B(\sss,\delta)$ and tends to $\sss$ as $r\to\infty$, provided $\xx_0$ in step MC0 of MC-Mode is sufficiently close to $\sss$.

For $r=0$, we have  $\xx_0=\sss^{(0)}_{n,k_0}\in B(\sss,\delta)$ by choice; therefore,
$$ \phi_0\leq \tau_nL^{2n}\delta\quad \Rightarrow\quad\text{$\phi_0<1$ provided $\delta$ sufficiently small.}$$
This implies that $\|\sss_{n,k_1}^{(1)}-\sss\|<\|\sss_{n,k_0}^{(0)}-\sss\|;$
therefore, $\sss_{n,k_1}^{(1)}\in B(\sss,\delta)$.
In addition, we also have $\phi_1<\phi_0$. Continuing by induction on $r$, we
see that $\phi_{r}<\phi_{r-1}<1$ hence  $\|\sss_{n,k_{r+1}}^{(r+1)}-\sss\|<
\|\sss_{n,k_{r}}^{(r)}-\sss\|$, which implies that $\sss_{n,k_{r+1}}^{(r+1)}\in B(\sss,\delta)$ since $\sss_{n,k_{r}}^{(r)}\in B(\sss,\delta)$, and that
$\lim_{r\to\infty}\|\sss_{n,k_r}^{(r)}-\sss\|=0,$  meaning that $\lim_{r\to\infty}\sss_{n,k_r}^{(r)}=\sss$.
With the convergence of $\{\sss_{n,k}^{(r)}\}^\infty_{r=0}$ to $\sss$ established, let us now rewrite \eqref{eq890} as
\beq  \frac{\|\sss_{n,k_{r+1}}^{(r+1)}-\sss\|}{\|\sss_{n,k_r}^{(r)}-\sss\|^2}
\leq \tau_nL^{2n}.\eeq
Taking the limsup as $r\to\infty$ on both sides of this inequality, we obtain \eqref{eq654d}.
Thus, the convergence of the sequence $\{\sss^{(r)}_{n,k_r}\}^\infty_{r=0}$ is quadratic.
\end{proof}

\section{\bf Remarks on $\|\SSS(\ee_n){}^+\|_2$} \label{se6}
\setcounter{equation}{0}
\setcounter{theorem}{0}
 Let us observe that  $\SSS(\yy)$ can be written as the product of two matrices as
\beq \SSS(\yy)=\PP\QQ(\yy),\eeq
where $\PP\in \C^{N\times kN}$ and $\QQ(\yy)\in \C^{kN\times k}$ are given as
\beq \PP=[\,\II\,|\,\tilde{\FF}\,|
\,\cdots\,|\,\tilde{\FF}{}^{k-1}\,];\quad \QQ(\yy)=\begin{bmatrix}\yy& \00& \cdots& \00\\
\00& \yy& \cdots& \00\\ \vdots& & \ddots& \vdots\\ \00& \00& \cdots& \yy\end{bmatrix}, \quad \yy\in\C^k.\eeq
Clearly, $\PP$ is a constant matrix and has full row rank,  while $\QQ(\yy)$ has full column rank for all nonzero $\yy$,  that is,
\beq \text{rank}(\PP)=N,\quad \text{rank}(\QQ(\yy))=k\ \ \forall\, \yy\neq\00.\eeq

Before going on, we recall that if $\KK\in\C^{m\times k}$ with $\text{rank}(\KK)=k$, then it has $k$ nonzero singular values, which we order such that
$$\sigma_1(\KK)\geq\sigma_2(\KK)\geq\cdots\geq \sigma_k(\KK)>0,$$ and
$$\sigma_k(\KK)=\min_{\zz\in \C^k,\ \|\zz\|_2=1}\|\KK\zz\|_2\quad\text{and}\quad \|\KK^+\|_2=1/\sigma_k(\KK).$$

Now, $\PP$ has $N$ positive singular values,
and therefore
$$ \|\PP^+\|_2=1/\sigma_N(\PP).$$
Next, $\QQ(\yy)$ is unitary when $\|\yy\|=1$, in the sense that
\beq \QQ(\yy)^*\QQ(\yy)=\II_{k\times k}\quad\forall\, \yy\in\C^k, \ \|\yy\|_2=1,\eeq
hence so is  $\QQ(\ee_n)$ since $\|\ee_n\|_2=1$. As a result $\QQ(\yy){}^+=\QQ(\yy)^*$ and $\QQ(\yy)$ has $k$ singular values, all equal to one, for all $\yy$, $\|\yy\|_2=1$.
Consequently,
\beq \|\QQ(\yy){}^+\|_2=1\quad \forall\,\yy\in\C^k, \ \|\yy\|_2=1.\eeq
Despite these interesting facts---that $\|\PP^+\|_2$ is fixed  and that $\|\QQ(\ee_n^{(r)}){}^+\|_2=1$  throughout the cycling process---we are not able to prove that
 $\|\SSS(\ee^{(r)}_n){}^+\|_2=\|[\PP\QQ(\ee^{(r)}_n)]{}^+\|_2\leq \alpha$ for some {\em fixed} $\alpha>0$,
 for all $r=0,1,\ldots,$
 where $\ee^{(r)}_n=\eepsilon^{(r)}_n/\|\eepsilon^{(r)}_n\|_2$
 in the $r$th cycle.

 For example, \eqref{eqA3} in the appendix, which would be extremely useful if applicable, does not apply
 to $\SSS(\yy)$. If it did, then we would have  $\SSS(\yy){}^+=\QQ(\yy)^*\PP{}^+$ hence  $\|\SSS(\yy)^+\|_2\leq\|\PP^+\|_2$, very conveniently.

 We might think that  Theorem \ref{thA4} in the appendix would apply to the $n$-Mode and C-Mode (it does not necessarily apply to the MC-Mode since the $\text{rank}(\SSS(\ee_n^{(r)}))=k_r$ may vary with $r$), but this too is problematic. Theorem \ref{thA4} requires the following:
 \begin{itemize}
 \item
 In the $n$-Mode, the sequence $\{\ee_n\}^\infty_{n=0}$, where $\ee_n=\eepsilon_n/\|\eepsilon_n\|_2$,  must have a limit $\ee_\infty$ such that $\text{rank}(\SSS(\ee_\infty))=k$. It is obvious from \eqref{eq37}--\eqref{eq38}
 that it is very difficult to determine whether such a vector $\ee_\infty$ exists when $\ff(\xx)$ is nonlinear.\footnote{For the linear system $\xx=\tilde{\ff}(\xx)$, we have  $\eepsilon_{n+1}=\tilde{\FF}\eepsilon_n$, $n=0,1,\ldots,$ as power iterations. Thus, in some cases,
 $\ee_\infty=\lim_{n\to\infty}\ee_n$ exists and   is an eigenvector of $\tilde{\FF}$,
 hence causes $\text{rank}(\SSS(\ee_\infty))=1$ at most. Clearly, this is a problem when $\text{rank}(\SSS(\ee_n))=k>1$, for $n=0,1,\ldots.$}
 \item
 In the C-Mode, the sequence $\{\ee_n^{(r)}\}^\infty_{r=0}$, where $\ee_n^{(r)}=\eepsilon_n^{(r)}/\|\eepsilon_n^{(r)}\|_2$,  must have a limit $\ee_n^{(\infty)}$ such that $\text{rank}(\SSS(\ee_n^{(\infty)}))=k$. It is obvious again from \eqref{eq37}--\eqref{eq38}  that it is very difficult to ascertain whether such a limit exists when $\ff(\xx)$ is  nonlinear.
\end{itemize}

A different approach to the issue, for the C-Mode, would be as follows:
Since $\SSS(\ee_n^{(r)})$ has full column rank, $\|\SSS(\ee_n^{(r)}){}^+\|_2=1/\sigma_k(\SSS(\ee_n^{(r)}))>0$ for every $r=1,2,\ldots.$
Defining the vector $\zzeta(\yy)\in \C^k$, $\|\zzeta(\yy)\|_2=1$,  via
\beq \min_{\zz\in\C^k,\, \|\zz\|_2=1} \|\SSS(\yy)\zz\|_2=\|\SSS(\yy)\zzeta(\yy)\|_2,\eeq
we thus have
\beq\sigma_k(\SSS(\ee_n^{(r)}))=\min_{\zz\in\C^k,\, \|\zz\|_2=1} \|\SSS(\ee_n^{(r)})\zz\|_2=\|\SSS(\ee_n^{(r)})\zzeta(\ee_n^{(r)})\|_2>0\quad \forall\, r=1,2,\ldots,\eeq
from which, we obtain
\beq  \sigma_k(\SSS(\ee_n^{(r)}))\geq
\liminf_{r\to\infty}\|\SSS(\ee_n^{(r)})\zzeta(\ee_n^{(r)})\|_2=\alpha\geq0. \label{eq559}\eeq
Clearly, $\alpha$ is independent of $r$.
Now, if we can  show that $\alpha>0$, we will have shown that $\|\SSS(\ee_n)^+\|_2\leq1/\alpha$, hence that $\|\SSS(\ee_n)^+\|_2$ is bounded uniformly throughout the cycling process.
Unfortunately, this does not seem to be the case in general; the best we can say is that $\alpha\geq0$.

Thus, even though $\sigma_k(\SSS(\ee^{(r)}_n))>0$  for $r=0,1,\ldots,$ it seems we cannot guarantee the existence of a fixed positive constant $\tilde{\alpha}$ such that, when applying RRE in the cycling mode, $\sigma_k(\SSS(\ee^{(r)}_n))\geq \tilde{\alpha}$ uniformly in every cycle. Therefore, we can only
{\em assume} that such a constant exists for the C-Mode cycling process being studied, for which $k$ is fixed throughout, namely,
\beq \label{eq234}\boxed{\|\SSS(\ee^{(r)}_n)^+\|_2\leq 1/\tilde{\alpha}<\infty\quad\forall\,r, \ \text{rank}(\SSS(\ee^{(r)}_n))=k\leq {k}_r, \quad r=0,1,\ldots,}\eeq
where ${k}_r$ is the degree of the minimal polynomial of $\tilde{\FF}$ with respect to $\ee^{(r)}_n$.

 As for the MC-Mode cycling process, we can, similarly,  only  {\em assume} that
\beq \label{eq235}\boxed{\|\SSS(\ee^{(r)}_n)^+\|_2\leq 1/\tilde{\alpha}<\infty\quad\forall\,r, \ \text{rank}(\SSS(\ee^{(r)}_n))={k}_r, \quad r=0,1,\ldots.}\eeq
(This is reasonable because there are only finitely many $k_r$ as $1\leq k_r\leq N$.)
Precisely \eqref{eq234} and  \eqref{eq235} are  what we have assumed in \eqref{eq662y}.

Finally, we note that the  global condition in \eqref{eq666}  we have imposed on the three modes for RRE  discussed in this work is formulated in terms of $\tilde{\FF}$, the Jacobian matrix of $\ff(x)$ at the solution $\sss$ only, and it  concerns $\sss_{n,k}$ with  arbitrary $n$. This should be contrasted with the global condition   introduced in \cite{Jbilou:1991:SRA} for the MC-Mode only that is formulated in terms of  $\ff(\xx)$, and concerns $\sss_{0,k}$.  Denoting the $\xx_i$ and the $\uu_i=\xx_{i+1}-\xx_i$ generated at the $r$th cycle by  $\xx_i^{(r)}$ and  $\uu_i^{(r)}$, respectively, with $\xx_0^{(r)}=\sss^{(r-1)}_{0,k_{r-1}}$,
the condition of \cite{Jbilou:1991:SRA} reads as follows:

$$ \sqrt{\det(\YY_r^* \YY_r)}\geq \alpha>0 \ \ \forall\, r; \quad  \YY_r=[\hat{\uu}_0^{(r)}\,|\,\hat{\uu}_1^{(r)}\,|\,\cdots\,|\,
\hat{\uu}_{k_r-1}^{(r)}\,],\quad \hat{\uu}_i^{(r)}={\uu}_i^{(r)}/\|{\uu}_i^{(r)}\|_2.$$

\appendix
\section*{Appendix: Some properties of Moore--Penrose inverses}
\setcounter{section}{1} \setcounter{equation}{0} \setcounter{theorem}{0}
First , we recall the well-known facts
\beq \AAA\in\C^{m\times n},\quad \text{rank}(\AAA)=n\quad\Rightarrow\quad \AAA^+=(\AAA^*\AAA)^{-1}\AAA^* \quad\Rightarrow\quad\AAA^+\AAA=\II_{n\times n},\eeq
\beq \AAA\in\C^{m\times n},\quad \text{rank}(\AAA)=m\quad\Rightarrow\quad
\AAA^+=\AAA^*(\AAA\AAA^*)^{-1} \quad\Rightarrow\quad\AAA\AAA^+=\II_{m\times m},\eeq
and
\beq \label{eqA3}\AAA\in\C^{m\times n},\quad \BB\in\C^{n\times p},\quad\text{rank}(\AAA)=\text{rank}(\BB)=n\quad
\Rightarrow\quad (\AAA\BB)^+=\BB^+\AAA^+.\eeq

The following theorems on Moore--Penrose inverses
of perturbed matrices can be found in Ben-Israel and Greville \cite{Ben-Israel:1966:EBG}, Wedin \cite{Wedin:1973:PTP}, and Stewart \cite{Stewart:1969:CGI}. Here we give independent proofs of two of  them.

\noindent{\bf Remark:} For convenience of notation, throughout this appendix only, we will use $\|\cdot\|$ to denote  the $l_2$ norm. (Thus,   $\|\cdot\|$ here does {\em not} stand for the $\GG$ norm we have used in Sections \ref{se1}--\ref{se6}.)
\begin{theorem} \label{thAA}
Let $\AAA\in\C^{m\times n}$, $\text{\em rank}(\AAA)=n$,  and let $\GG\in\C^{m\times m}$ be nonsingular and define $\BB=\GG\AAA$. Then $\text{\em rank}(\BB)=n$ too,  and
$$ \|\BB^+\|\leq \|\GG^{-1}\|\|\AAA^+\|.$$
\end{theorem}
\begin{proof} That $\text{rank}(\BB)=n$ is clear since $\GG$ is nonsingular. Starting now with $\AAA=\GG^{-1}\BB$, we first have
$$ \AAA\xx=\GG^{-1}(\BB\xx)\quad \Rightarrow \quad \|\AAA\xx\|
\leq\|\GG^{-1}\|\,\|\BB\xx\|\quad \forall\, \xx\in\C^n,\quad \|\xx\|=1.$$
Let $\xx'$ and $\xx''$, with $\|\xx'\|=1$ and $\|\xx''\|=1$, be such that $$\sigma_{\min}(\AAA)=\min_{\|\xx\|=1}\|\AAA\xx\|=\|\AAA\xx'\|\quad\text{and}\quad  \sigma_{\min}(\BB)=\min_{\|\xx\|=1}\|\BB\xx\|=\|\BB\xx''\|,$$
where $\sigma_{\min}(\KK)$ denotes the smallest singular value of a matrix $\KK$.
Then
$$\sigma_{\min}(\AAA)=\|\AAA\xx'\|\leq \|\AAA\xx''\|\leq \|\GG^{-1}\|\,\|\BB\xx''\|
=\|\GG^{-1}\|\,\sigma_{\min}(\BB). $$
The result follows by recalling that $\|\KK^+\|=1/\sigma_{\min}(\KK) $ when $\KK$ has full column rank, which implies that
$\sigma_{\min}(\KK)>0$.
\end{proof}

\begin{theorem}
\label{thA1} Let $\AAA\in\C^{m\times n}$ and $(\AAA+\EE)\in\C^{m\times n}$, $m\geq n$, such that
$\text{\em rank}(\AAA)=n$ and $\|\EE\AAA^+\|<1.$
Then
$$\|(\AAA+\EE)^+\|\leq\frac{\|\AAA^+\|}{1-\|\EE\AAA^+\|}.$$
If $\Delta=\|\EE\|\,\|\AAA^+\|<1$ in addition, then this result can be expressed as
$$\|(\AAA+\EE)^+\|\leq\frac{1}{1-\Delta}\|\AAA^+\|.$$
\end{theorem}

\begin{proof}
First, because  $\AAA$ is of full column rank, we have that $\AAA^+\AAA=\II_{n\times n}$.
Consequently,
$$\AAA+\EE=(\II+\EE\AAA^+)\AAA.$$
Since $\|\EE\AAA^+\|<1$ by assumption, the matrix $\GG=\II+\EE\AAA^+$ is nonsingular.
The first result now follows from Theorem \ref{thAA} and by the fact that $\|\GG^{-1}\|\leq 1/(1-\|\EE\AAA^+\|)$. The second result follows by invoking  $\|\EE\AAA^+\|\leq \|\EE\|\,\|\AAA^+\|=\Delta$ and the additional assumption that $\Delta<1$.
\end{proof}

\begin{theorem} \label{thA2}
Let $\AAA$ and $\EE$ be as in Theorem \ref{thA1}, $\Delta=\|\EE\|\,\|\AAA^+\|<1$, and let
$\HH=(\AAA+\EE)^+-\AAA^+$. Then
$$ \|\HH\|\leq\sqrt{2}\frac{\Delta}{1-\Delta}\|\AAA^+\|.$$
\end{theorem}
\begin{proof}
By Wedin \cite[Theorem 4.1]{Wedin:1973:PTP},  there holds

$$\|\HH\|\leq\sqrt{2}\,\|(\AAA+\EE)^+\|\,\|\AAA^+\|\,\|\EE\|.$$
Invoking now Theorem \ref{thA1}, the result follows.
\end{proof}

The following theorem is due to Stewart \cite{Stewart:1969:CGI}.
\begin{theorem} \label{thA4}
Let $\AAA_1,\AAA_2,\ldots,$ and $\AAA$ be such that $\lim_{n\to\infty}\AAA_n=\AAA$. Then
$\lim_{n\to\infty}\AAA_n^+=\AAA^+$ if and only if $\text{\em rank}(\AAA_n)=\text{\em rank}(\AAA)$, $n\geq n_0$,  for some integer $n_0$.
\end{theorem}
 \section*{Acknowledgement} The author would like to thank one of the anonymous referees for his/her remarks that helped to improve the presentation and results of this work substantially.

\end{document}